\documentclass[a4paper,10pt,twoside]{article}
\usepackage[utf8]{inputenc} % permette di scrivere gli accenti della tastiera su LINUX
\usepackage{graphicx}
\usepackage{authblk}
\usepackage[pdftex,bookmarks]{hyperref}%%indice dinamico
\usepackage{amsmath}    %scrittura caratteri matematici
\usepackage{amssymb}
\usepackage{amsmath,amsthm}
\usepackage{textcomp}
\usepackage{color}
\usepackage[all]{xy}
\usepackage{amsmath,amsthm,amsfonts,latexsym,amscd,amssymb,enumerate,wasysym,stmaryrd,skull,}

\usepackage{fancyhdr}
\definecolor{light}{gray}{.95}

\fancypagestyle{plain}{
	\fancyhead{}
	
}

\usepackage[T1]{fontenc}
\usepackage{sansmath}
\sansmath

\numberwithin{equation}{section}

\theoremstyle{plain}% default

\newtheorem{teo}{Theorem}[section]
\newtheorem{lem}[teo]{Lemma}

\theoremstyle{definition}

\newtheorem{ddef}[teo]{Definition}

\newtheorem{rem}[teo]{Remark}

\renewenvironment{proof}%
{\par \vspace{0.2cm}\ \\ \noindent{ { \textsc{Proof}}\,---\ } }{\hfill{$\Box$} \par \vspace{0.2 cm}}

\theoremstyle{remark}

\newcommand{\RR}{\mathbb{R}}
\newcommand{\HH}{\mathbb{H}^n}
\renewcommand{\tilde}{\widetilde}

\newcommand{\Cs}{$C^*$-}

\begin{document}

\title{The Calderon projection over $C^*$--algebras}
\author{Paolo Antonini\thanks{ 
Projet Alg\'ebres d'op\'erateurs
Universit\'e Paris Diderot paolo.anton@gmail.com}}

\maketitle
\begin{abstract}We construct the Calderon projection on the space of Cauchy datas for a twisted Dirac operator on a compact manifold with boundary acting on a bundle of finitely generated $C^*$--Hilbert modules . In particular an invertible double is constructed in the Mischenko--Fomenko setting generalizing the classical result.
\end{abstract}
\section{Introduction}
The introduction of \Cs algebras in index theory and differential geometry initiated by Mishchenko, Fomenko, Connes, Kasparov and Moscovici 
\cite{connes,mf,kas} led to a number of applications and new insights including the establishment of the Novikov conjecture for a large class of manifolds \cite{cm} or the Connes--Skandalis general index theorem for foliations \cite{connesskandalis}. After this appearance, primary and secondary invariants of elliptic operators gained a promotion. 
Higher indices belong to the $K$--theory of a ground $C^*$--algebra while the ordinary (numerical) ones are called lower indices.
Higher invariants contain refined informations and gain stability properties from the cohomological character of the $K$-theory of operator algebras.
For example it is well known that the $C^*$--algebraic index class of the signature operator is homotopy invariant.

While the theory of elliptic operators which are invariant under the action
of \Cs algebras is nowadays well founded for closed manifolds, less is known for structures with boundary. There are at least two methods to deal with geometric operators on a manifold with boundary, doing analysis on the incomplete manifold following the paradigm of elliptic boundary value problems, or following the Melrose $b$--philosophy by looking at the associated complete manifold with cylindrical ends.
The two points of view should not be intended complementary or opposite but  integrating one each other. This is clear from the beginning and pointed out in the seminal paper by Atiyah Patodi and Singer \cite{aps}. The interplay complete/incomplete becomes essential when dealing with the topological properties of signature operator which was the original motivation of A.P.S.

 The cylindrical case in the higher setting was studied by Piazza, Leichtnam, Schick, Lott and Wahl \cite{lp1,lp2,charlotte1}. In the incomplete case the literature is still lacking. Since the spectrum of the boundary operator is no more discrete there are, in general, no A.P.S. boundary conditions in strict sense. The formulation of global elliptic boundary value problems relies on the notion of noncommutative spectral section. 
 
 In this paper we generalize the classical theory constructing a basic tool for the investigation of boundary value problems of Dirac operators acting on sections of bundles of finitely generated Hilbert modules over a $C^*$ algebra (tipically In the applications is the $C^*$--algebra of the fundamental group).
We show the existence of a nice operator called the Calderon projection. It is an order zero pseudodifferential operator in the Mishchenko--Fomenko calculus on the boundary projecting on the space of the smooth Cauchy datas. In the classical situation the whole theory is mastered by the property of this operator. Indeed not only the definition of ellipticity for boundary value problems is expressed in terms of the Calderon projection but also the index of the Fredholm realization can be computed as the relative index of the projection and the boundary condition.  We postpone applications to a future paper.
\subsubsection*{Acknowledgments}
It is a pleasure to thank Paolo Piazza for having proposed this line of research, Georges Skandalis, Francesca Arici and Charlotte Wahl for interesting discussions.

 \section{Review of the classical theory}
  \subsection{Unique continuation property}
\begin{ddef}
One says that an operator $A$ on a smooth connected manifold $M$ (also with boundary) has the \emph{unique continuation property} (U.C.P.) if every solution $s$,
$$As=0$$ which vanishes on an open set also vanishes everywhere.
\end{ddef}

It is well known that all the Dirac type operators on a manifold enjoy the U.C.P. There is a huge amount of literature on the subject. We limit ourselves to cite the exaustive book \cite{Boos} and the recent paper 
\cite{Lesch}. The crucial property of Dirac type operators $D$, which moreover distinguishes them among first order ones, is the product form:
\begin{equation}
\label{productform}
D=G(y,u)(\partial_u+B_u),
\end{equation}
for a locally deformed Riemannian structure. This is true a fortiori on a manifold with boundary \cite{aps}, with product metric. The tangential piece $B_u$ has an \underline{elli}p\underline{tic} and selfadjoint part $1/2(B_u+B_u^*)$.

For a Dirac type operator the Green formula reads:
\begin{equation}
\label{green}
\langle D s_1, s_2 \rangle - \langle s_1, D^* s_2 \rangle = - \int_{\partial M} \langle c(v)(s_1 |_{\partial M}), s_2 |_{\partial M} \rangle.
 \end{equation}
Here $c(v)$ denotes Clifford multiplication by the inward normal unit vector to $\partial M$.
Let us emphasize that on a manifold with boundary, if the metric and all the geometric datas defining the Dirac operator are product type near the boundary, then equation \eqref{productform} simplifies and the operator writes as
\begin{equation}
\label{prodformbnd}
D=G(y)(\partial_u+B),
\end{equation} 
in a collar neighborhood $N$ of the boundary, with $G$ unitary and $B$ selfadjoint, both independent on the normal variable $u$. In this case the unique continuation property near the boundary immediately follows from elementary harmonic analysis by expanding the solution in the form
 \begin{equation}\label{spectral}
 s(u,y)=\sum_{\lambda}f_{\lambda}(u)\varphi_{\lambda}(y)
 \end{equation} where ${\varphi_{\lambda}(y)}_{\lambda}$ is a spectral resolution of $B$ over the boundary. The expansion \eqref{spectral} also plays a crucial role in the original proof of the Atiyah--Patodi--Singer index formula, in particular it establishes the equivalence of the A.P.S. pseudodifferential elliptic boundary value problem with the natural $L^2$ theory on the corresponding manifold with an attached cylinder.
 
The typical application of this principle, (from \eqref{green}) is the following result
\begin{teo}Let $X=X_+ \cup X_-$ be a connected partitioned manifold with $X_+\cap X_-=\partial X_{\pm}=Y.$ Then there are no smooth \emph{ghost solutions} i.e. solutions $s$ of $As=0$ such that $s_{|Y}=0.$
\end{teo}
We rapidly review the proof of the classical unique continuation property. If $s$ is zero on an open set $V$ which is properly contained in $M$ one choose some $x_0\in \partial V$ and a point $p\in V$ at distance $r$ from $x_0$ such that the ball $B(p,r)$ is contained in $V$. 
One shows that $$s_{|B(p,r+T/2)}=0\quad \textrm{for some }T>0$$
In turn this follows from the Carleman estimate \cite{Boos}, holding for every arbitrary sufficiently big $R>0$
\begin{equation}
\boxed{
 R\int_{u=0}^{T}\int_{\mathbb{S}^{n-1}_{p,u}}e^{R(T-u)^2}\|v(u,y)\|^2dydu\leq C\int_{u=0}^T\int_{\mathbb{S}^{n-1}_{p,u}}\|Av(u,y)\|^2dydu}
 \end{equation} with $v(u,y):=\varphi(u)s(u,y)$ for a smooth cut off function $\varphi$ such that $\varphi_{|u<8/10T}=1$ and $\varphi_{|u>9/10T}=0.$
 The unique continuation property for Dirac operators also holds in a $C^*$ algebraic setting:
 
 \begin{teo}
 Let $X$ be a compact manifold equipped with a Clifford module bundle $\mathcal{E}$ endowed with compatible connection and $\mathcal{A}$ a bundle of finitely generated projective Hilbert modules over a unital $C^*$ algebra\footnote{we assume our $C^*$ algebras to be complex; however everything can be formulated for real $C^*$ algebras under small modification} $A$. One forms the twisted Dirac operator 
 $$D:\Gamma(X;\mathcal{E}\otimes \mathcal{A})\longrightarrow  \Gamma(X;\mathcal{E}\otimes \mathcal{A}).$$
 If a smooth section $s$ satisfies $Ds=0$ and vanishes on an open set $\Omega \subset X$ then $s=0$.
 \end{teo}
 \begin{proof}It immediately  follows from the fundamental estimate of Xie and Yu \cite{Xie}
 $$\|s\|\leq C_1(\Omega)\|s_{|\Omega}\|+C_2(\Omega)\|Ds\|.$$
 \end{proof}

 In this section we briefly recall, for the convenience of the reader, the properties of the standard Calderon projector.
 There is a lot of excellent literature on the subject \cite{Boos,Lesch,Lesch2,calderon}. We refer to it.
 
 \subsection{The invertible double construction}
 \label{double}
 Let us assume for simplicity that $X$ is a compact \emph{even} dimensional manifold with bounday $\partial X =Y$. This assumption will allow us to make use of the chiral notation. We denote by $D^+:\Gamma(X;\mathcal{E}^+)\longrightarrow \Gamma(X;\mathcal{E}^-)$ the positive part of a chiral Dirac operator on $X$. Assume that all structure are product near the boundary. Then, $D^+$ is in product form \eqref{prodformbnd} with unitary Green form $G$ and selfadjoint tangential part $B$. 

One can construct the doubled manifold $\widetilde{X}:=X_1 \cup_{c(du)}X_2$ attaching a reversed copy of $X$ by the Clifford multiplication along the boundary. More precisely, let $X_1=X$ and denote by $X_2:=-X$, the manifold with opposite orientation. We attach a cylinder $\partial Y \otimes (-\epsilon,0]$ to $X_1$ and $[0, \epsilon) x \partial Y$ to $X_2$, glue together the two resulting manifolds and send $\epsilon$ to zero.

The operators $G$ and $B$ anticommute, since $D$ is (formally) selfadjoint. Moreover, the unitary map sending a section $f \in \Gamma ([0,\epsilon) \times \partial X)$ to $Gf \in \Gamma((-\epsilon,0]\times \partial X )$ conjugates $D$ and $-D$.

We can glue together $D^+$ and $-D^+$ into a new operator
\[\tilde{D^+} := D^+ \cup_G -D^+\]
using $J$ as a clutching function. 

Let use introduce the corresponding bundles, $\mathcal{E}^+$ over $X_1$ and $\mathcal{E}^-$ over $X_2$ using the identification given by Clifford multiplication. Then we obtain the bundles of spinors of positive and of negative chirality

\[ \tilde{\mathcal{E}}^+ := \mathcal{E}^+ \cup_{c(v)}\mathcal{E}^- \qquad \tilde{\mathcal{E}}^- := \mathcal{E}^- \cup_{c(v)} \mathcal{E}^+.\]

\begin{figure}[h!]
\centering
  \includegraphics[width=4in]{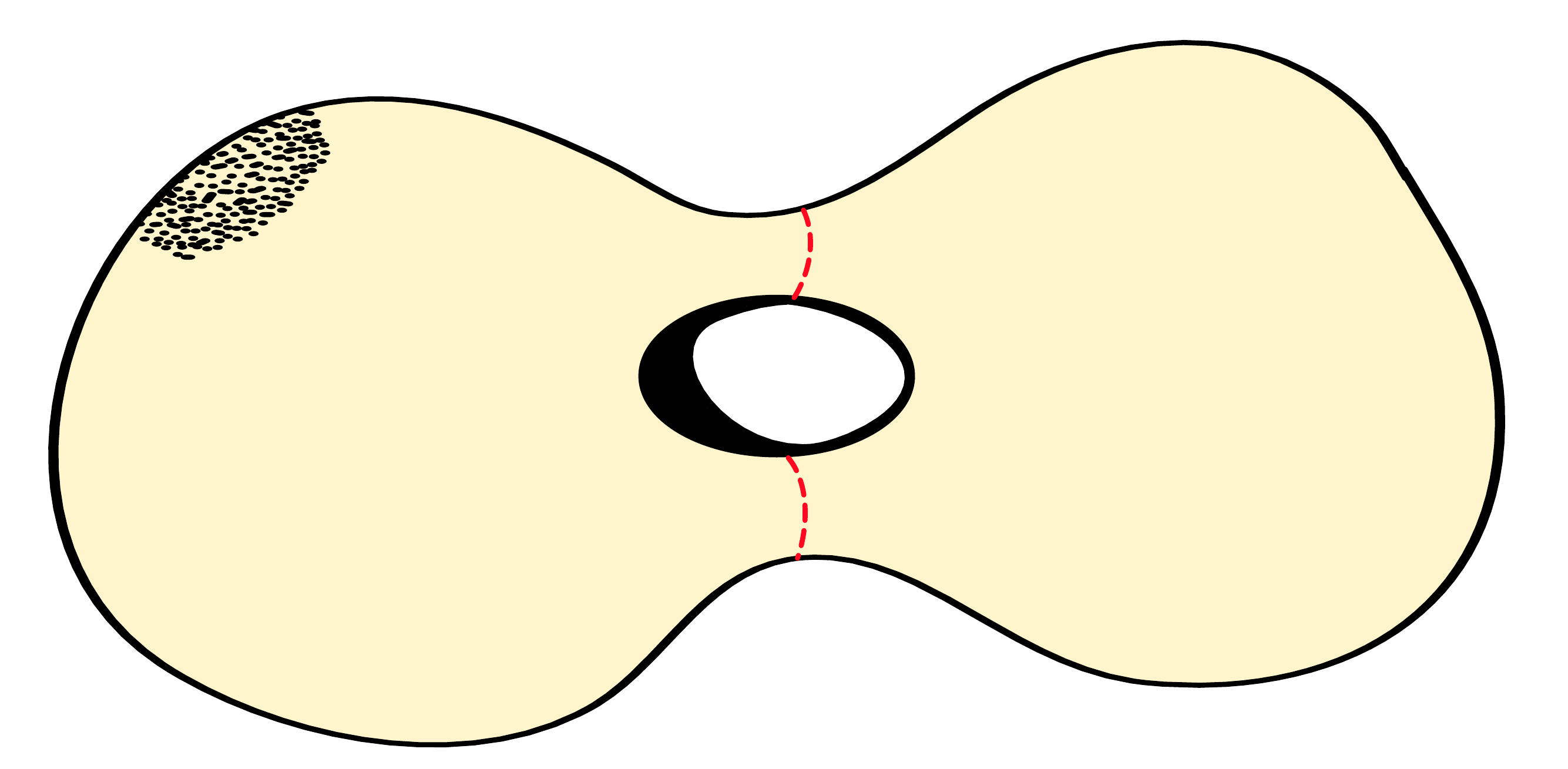}
  \caption[The doubled manifold.]
   {The doubled manifold.}
\end{figure}

So both the bundles and the operator extend to a doubled operator. The opeator is invertible due to the unique continuation property; the inverse is of course pseudodifferential of order $-1$. 

An analogous construction is performed for the negative part of the chiral Dirac opeator $D^-$ to get the total Dirac opeator $\tilde{D}$.
\subsection{The Calderon projector}

%\subsubsection{Sobolev spaces on manifolds with boundary}
%The following chains of sobolev spaces are naturally associated to our data. Let $s \in \RR, s >0$
%\begin{enumerate}
%\item The space $H^s(\tilde{X})$ of complex $L^2$-functions over $X$ which yield elements in $H^s(\RR^n)$ in local coordinates;
%\item The space $H^s(Y)$, defined as above;
%\item The space $H^s(X)$ consisting of restrictions $\lbrace r^+ u \ \vert \ u \in H^s(\tilde{X}) \rbrace$, with $r^+: L^2(\tilde{X}) \rightarrow L^2(X)$ being the restriction operator.
%\end{enumerate}
%Let us remark that the Sobolev norms, even though depending on the choice of smooth atlas for $\tilde{X}$, are all equivalent and yield the same topology.
%\begin{teo}
%For $u \in \Gamma(\tilde{X})$, we define $u_t = u (\psi(y,t))$ for a fixed parametrization $\psi: [-\epsilon , \epsilon] \times Y$ of a collar neighborhood of $Y$. Then for any $t \in (-\epsilon, \epsilon)$ the mapping $u \mapsto u_t$ provides a continuous map from $\Gamma (\tilde{X})$ to $\Gamma (Y)$ that can be extended for any $s > \frac{1}{2}$ to a continuous linear map 
%\begin{equation}
%\label{tracemap}
%\gamma_{t} : H^s(\tilde{X}) \rightarrow H^{s-1/2}(Y).
%\end{equation}
%\end{teo}

\subsubsection{Spaces of Chauchy data}
The following definition holds both for the total Dirac opeator $D$ and for its chiral compontent $D^+$.  
\begin{ddef}
Let $D: \Gamma(M, E) \rightarrow \Gamma(M, F)$
be a Dirac type operator over a partitioned manifold $X= X_1 \sqcup X_2$, $ X_1 \cap X_2 = Y$.
We define the spaces of Cauchy datas $H_{1,2}$ to be
\begin{equation}
H_{i}(D) := \lbrace u \vert Y \ \vert \ u \in \Gamma (E) , \ Du =0 \mbox{ on } X_{i} \rbrace \quad i=1,2.
\end{equation}
\end{ddef} 
These are the spaces of the traces on the boundary of smooth solutions in $X_1$ or $X_2$. 

It is easy to see from the Green formula that the space of Cauchy datas $H_1(\tilde{D})$ and $H_2(\tilde{D})$ intersect only in the zero section and are $L^2$ orthogonal. 

Let $r_{1,2}$ denote the restriction maps, $r_{i}(f_1, f_2) = f_{i}$, and $\gamma_0^{\pm}$ the trace map.
Composing the adjoint $\gamma^*$ of the trace map $\gamma_-$  with the inverse of the operator, and restricting to $X_i$ one forms the Poisson operator:
\[K_{i}:= r_i(\tilde{D})^{-1}(\gamma_0^-)^*G .\] 
The Calderon projector $\mathcal{C}_i$ is defined as the trace to the boundary composed with the Poisson operator.

Since there are non--trivial regularity issues a lot of work has to be done to show that these traces make sense and the limits in a collar of the restrictions converge in $L^2$. The delicate analysis carried out in \cite{Boos,calderon} that we shall repurpose next for \Cs bundles shows that $\mathcal{C}^+$ is a pseudodifferential idempotent projection on $H_1$ along $H_2$. 

The principal symbol of the Calderon projector is the same of the A.P.S. boundary condition i.e. the projection on the space of eigenvalues of the principal symbol of $D^+$ with positive real part. Indeed from this coincidence one can develop a theory of global elliptic boundary value problems. The boundary conditions must satisfy an assumption on the principal symbol which is formulated in terms of the principal symbol of $\mathcal{C}^+$. One of the most important results is the possibility of expressing the index of the boundary value problems in terms of purely boundary operators. More precisely if $R$ is such a boundary condition, then the $L^2$--realization of $D^+$ i.e. the unbounded operator $D_R^+$ acting on
$$\operatorname{Dom}(D_R^+):=\{u\in H^1(X):\,R(u_{|Y})=0\}$$ is Fredholm and 
$$\operatorname{ind}(D^+_R)=i(R,\mathcal{C}_+).$$ At the right--hand side we find the relative index of two projections \cite{Boos}. We plan to develop a similar formula in the context of \Cs bundles in a forthcoming paper \cite{noi}.
 % We recall the Green formula.
 %\begin{teo}
%Let $D$ as above. Then the following equality holds:
%\begin{equation}
 %\langle Ds_1,s_2\rangle_A-\langle s_1,D^* s_2\rangle_A=-\int_{X}\langle c(v)(s_1)_{|\partial X},(s_2)_{|\partial X}\rangle_A dy
%\end{equation}
%where $c(v)$ denotes Clifford multiplication by the inward normal unit vector to the boundary of $X$.
 %\end{teo}

 \section{Sobolev modules}
 Let $X$ be a compact Riemannian manifold with boundary $Y:=\partial X$, let $\widetilde{X}$ the double manifold. Given a unital \Cs algebra $A$, we define bundles $\mathcal{A},\,\, \widetilde{\mathcal{A}}$ of finitely generated projective Hilbert $A$--modules over $X$ and $\widetilde{X}$ respectively.
 
For $k\in \mathbb{N}$ the Hilbert--Sobolev modules of sections are defined as a discrete chain of topological \Cs -Hilbert $A$ modules \cite{Va}, in particular there is no preferred Hilbertian product on the $\mathcal{H}^k$'s - except for $\mathcal{H}^0$ - but an admissible class of products such that the induced Banach topologies are the same and give rise to the same space of adjointable functionals. 
Admissible Hilbert products can be defined in coordinate patches for $X,\,\widetilde{X}$, using the notion of weak derivativatives ($L^2$--derivatives) or, for $\widetilde{X}$ using powers of the Laplacian \cite{Va}.
 Let $\mathbb{H}^n:=\{x \in \mathbb{R}^n:\, x_n\geq 0\}$ denote the half space and let $V$ be a Hilbert $A$-module. For real $s$ the Hilbert module $\mathcal{H}^s(\RR^n;V)$ can be defined by the Fourier transform 
 $$\mathcal{H}^s(\RR^n;V)=\{f \in \mathcal{H}^0(\RR^n;V):\, (1+|\xi|^2)^{s/2}\hat{f}(\xi)\in \mathcal{H}^0(\RR^n;V)\}$$ where
 $$\hat{f}(\xi):=\frac{1}{(2\pi)^{\frac{n}{2}}}\int e^{-i \xi \cdot x}f(x) dx.$$
 More generally, following Schwartz \cite{Schwartz}, we can define (tempered) distributions with values in $V$. These will be useful next.
  We have a continuous embedding with dense range $$J_{\RR^n}:\mathcal{H}^1(\RR^n;V)\longrightarrow \mathcal{H}^0(\RR^n;V).$$ 
The Hilbert--module adjoint is easily described in terms of the Fourier transform. If $h\in \mathcal{H}^0(\RR^n;V)$ then
 $$\widehat{J_{\RR^n}^*h}=\dfrac{\hat{h}(\xi)}{1+|\xi|^2}\, \textrm{ i.e. }J_{\RR^n}^*h=\Delta^{-1}h\,.$$ From this formula we see that $J^*_{\RR^n}$ is surjective on the Schwartz sections and has dense range.
 Since we are dealing only with $L^2$ and $H^1$ we can define a continuous extension operator by reflection $\ell_i:\mathcal{H}^i(\HH;V)\longrightarrow \mathcal{H}^i(\RR^n;V)$, $i=0,1$, having a more simplified expression than the classical one (\cite{Boos}):
 \begin{equation}\label{ext}
\ell_i f:=\begin{cases}
f(y,t) \quad t\geq 0
\\
-f(y,-t) \quad t<0
\end{cases}
\end{equation} 
It has an obvious section which is the restriction map. It is adjointable with adjoint map $\ell_i^*:g\longmapsto g(y,t)+g(y,-t)$.
We have a commutative diagram,
$$\xymatrix{\mathcal{H}^1(\RR^n;V)\ar[rr]^{J_{\RR^n}}& &\mathcal{H}^0(\RR^n;V)\ar[d]^{r_0}\\\mathcal{H}^1(\HH;V)\ar[rr]_{J_{\HH}}\ar[u]^{\ell_1}& &\mathcal{H}^0(\HH;V)}$$ then $r_0$ is the extension by zero and
$$J^*_{\HH}=\ell_1^*\circ J^*_{\RR^n}\circ r_0^*.$$ Now introduce in $\RR^n$ the operators $\Lambda_{\pm}:=\mp \partial_{x_n}+\sqrt{1+\Delta_{(n-1)}}$. They are isomorphisms from $\mathcal{H}^s$ to $\mathcal{H}^{s-1}$, $\Lambda_+^*=\Lambda_-$ and $\Lambda_+\Lambda_-=\Delta_{(n)}=\Delta$ the positive Laplace Beltrami operator. It is well known that for every $t$, the operator $ \Lambda_+^t$ preserves the space of distributions supported in the half space $\mathbb{H}^n_-$ \cite{Boos} and  $\Lambda_-$ preserves distributions supported in $\HH_+$. Now applying this together with the property $\star \Lambda_+ \star=\Lambda_+$ for $\star:f\longmapsto -f(-t)$ one sees that, when restricted to Schwartz sections, the map $\ell_1^*$ is simply the restriction and $J_{\RR}^*$ can be inverted. In other words $J_{\RR}^*$ has dense range.

In particular by the standard Friedrichs method \cite{Ta} $\mathcal{H}^1(\HH)$ sitting dense in $\mathcal{H}^0(\HH)$ is the domain of a positive selfadjont operator $D$ and the complex interpolation procedure can be carried on to define the intermediate Hilbert--Sobolev modules
$$\mathcal{H}^{\theta}(\HH;V):=[\mathcal{H}^0(\HH;V),\mathcal{H}^1(\HH;V)]_{\theta}=\mathcal{D}(D^{\theta}),\quad 0<\theta<1.$$
Negative order spaces are defined by duality as usual.
These results are transported on a Riemannian manifold with boundary using coordinates, local trivializations and a partition of unity. 
   \section{The Calderon projection}

 \subsection{Invertible double}
Following the recent work \cite{Xie} and the classical theory \cite{Boos, Lesch} we carry on the construction of the invertible double of the Dirac operator coupled with \Cs Hilbert module bundles. 
We do everything in even dimension just for notational simplicity. All the results of this section extend to odd dimensions in a trivial manner.

So let $X_1$ be an even dimensional compact Riemannian manifold with boundary $Y$. Let $\mathcal{E}_1$ be a graded Clifford bundle over $X_1$ and $\mathcal{A}$ a bundle of finitely generated projective Hilbert $A$-modules for a (real or complex) unital $C^*$ algebra $A$. Assume the metric and the Clifford structure are product near the boundary. For simplicity, we will also assume that the connection and metric on the twisting bundles to be in product form.

We denote by $X_2:=-X_1$ the manifold with opposite orientation and corresponding Clifford bundle $\mathcal{E}_2$. The bundles $\mathcal{E}_1\otimes \mathcal{A}$ and $\mathcal{E}_2\otimes \mathcal{A}$ are glued together by the Clifford multiplication $c(v)$ where $v:=d/du$ is the inward unit normal vector near the boundary of $X_1$. We decorate with tilde the resulting doubled bundle and manifold i.e. $\widetilde{X}:=X_1 \cup_{c(v)}X_2$ and
$$\widetilde{\mathcal{E}}\otimes \widetilde{\mathcal{A}}:=\big{(}\mathcal{E}_1^{\pm}\cup_{c(v)}\mathcal{E}_2^{\mp}\big{)}\otimes \mathcal{A}.$$
A section of $\widetilde{\mathcal{E}}^+\otimes \widetilde{\mathcal{A}}$ can be identified with a pair $(s_1,s_2)$ where $s_1\in \Gamma(\mathcal{E}_1^+\otimes \mathcal{A})$, $s_2\in \Gamma(\mathcal{E}_2^-\otimes \mathcal{A})$ such that near the boundary $$s_2=c(v)s_1.$$ 

We have two graded  Dirac operators 
$$D_i^{\pm}:\Gamma(X_i;\mathcal{E}_i^\pm \otimes \mathcal{A})\longrightarrow \Gamma(X_i;\mathcal{E}_i^\mp\otimes \mathcal{A})$$ and a resulting double operator
$\widetilde{D}$ on $\widetilde{X}$,
$$\widetilde{D}^{\pm}(s_1,s_2):=(D^{\pm}_1s_1,D_2^{\mp}s_2).$$ It is a bounded operator between the corresponding Hilbert--Sobolev modules
$$
\left(\begin{array}{cc}0 & \widetilde{D}^- \\\widetilde{D}^+ & 0\end{array}\right):\mathcal{H}^1(\widetilde{X};\widetilde{\mathcal{E}}^+\otimes \mathcal{A})\oplus \mathcal{H}^1(\widetilde{X};\widetilde{\mathcal{E}}^-\otimes \mathcal{A})\longrightarrow \mathcal{H}^0(\widetilde{X};\widetilde{\mathcal{E}}^+\otimes \mathcal{A})\oplus \mathcal{H}^0(\widetilde{X};\widetilde{\mathcal{E}}^-\otimes \mathcal{A})
$$
\begin{teo}
({\bf{Invertible double construction}}). 
The operator $$\widetilde{D}^+:\mathcal{H}^1(\widetilde{X};\mathcal{\widetilde{E}}^+\otimes \widetilde{\mathcal{A}})\longrightarrow \mathcal{H}^0(\widetilde{X};\widetilde{\mathcal{E}}^-\otimes \widetilde{\mathcal{A}})$$ is invertible with bounded inverse 
$$(\widetilde{D}^+)^{-1}:\mathcal{H}^0(\widetilde{X};\widetilde{\mathcal{E}}^-\otimes \widetilde{\mathcal{A}})\longrightarrow \mathcal{H}^1(\widetilde{X};\widetilde{\mathcal{E}}^+\otimes \widetilde{\mathcal{A}}).$$
\end{teo}
\begin{proof}First of all the operator $\widetilde{D}^+$ is bounded below because the entire $\widetilde{D}$ is bounded below by Theorem 5.1 in \cite{Xie}: 
$$\|\sigma\|\leq C\|\widetilde{D}\sigma\|,\quad \sigma\in \Gamma(\widetilde{X};\widetilde{\mathcal{E}}\otimes \widetilde{\mathcal{A}}) .$$
Then zero is isolated in the spectrum of $(\widetilde{D}^+)^*\widetilde{D}^+$ by the John Roe convergence transfer principle (Proposition 1.13 in \cite{roe}). Indeed the original proof works word by word in the context of Hilbert modules. Then by Lemma 3.2 in the Appendix and by the Mishchenko Lemma we know that $\operatorname{Ran}(\widetilde{D}^+)$ is closed and orthocomplemented 
$$\mathcal{H}^0(\widetilde{X};\widetilde{E}^-\otimes \widetilde{\mathcal{A}})=\operatorname{Ker}\widetilde{D}^-\oplus^{\bot} \operatorname{Ran}\widetilde{D}^+.$$ It is sufficient to show that $\operatorname{Ker}\widetilde{D}^-$ is zero. This follow as in the classical situation by the unique continuation property. Indeed let $(s_1,s_2)$ such a solution. We are going to show that it vanishes on $Y$. Indeed $D^+s_1=0=D^-s_2$ hence
$$0=\langle D^+s_1,s_2\rangle_A-\langle s_1,D^-s_2\rangle_A=-\int_{X}\langle c(v)(s_1)_{|Y},(s_2)_{|Y}\rangle_A dy=\langle (s_2)_{|Y}, (s_2)_{|Y}\rangle_A.$$
It follows exactly by the classical argument (Lemma 9.2. in \cite{Boos}) that setting
\begin{equation}\label{uep}
\widetilde{s}:=\begin{cases}
s_1 \quad \operatorname{on }X_1
\\
0 \quad \operatorname{on }X_2
\end{cases}
\end{equation} gives a weak solution and by elliptic regularity a strong solution. This solution must be zero by the unique continuation property.
\end{proof} 
\begin{rem}The invertible double construction remains valid if the manifold is no more compact but the scalar curvature is positive and bounded by below outside a compact set \cite{Xie}.
\end{rem}
Of course the inverse $(\widetilde{D}^+)^{-1}$ is a order $-1$ pseudo differential operator in the Mishchenko--Fomenko calculus. To see that just take a parametrix $Q$,
$$Q\widetilde{D}^+=1+K$$ with $K$ a smoothing operator. Then
$$(\widetilde{D}^{+})^{-1}=Q-K(\widetilde{D}^{+})^{-1}.$$
\subsection{The Calderon projection}
Define the space of Cauchy datas along $Y$ by
$$H_{i}(\widetilde{D}^+):=\{u_{|Y}: u\in \Gamma
(\widetilde{\mathcal{E}}^+\otimes \mathcal{A})
,\quad \widetilde{D}^+u=0, \operatorname{ in } X_{i}\}, \quad i=1,2.$$
The closure of the Cauchy data spaces in $\mathcal{H}^{s-1/2}$ will be denoted by $H_{1,2}(\widetilde{D}^+,s).$
 It is easy to see from the Green formula that these spaces intersect only in the zero section. 
We denote, for $s>1/2$ by $\operatorname{Ker}_{1,2}(\widetilde{D}^+,s)$ the closure in $\mathcal{H}^{s-1/2}$ of the kernel of $\widetilde{D}^+$ in $X_{1,2}$.
By the unique continuation property there are no solutions in these spaces with support contained in the interior or identically vanishing on the boundary. Let $r_1$ be the operator of restriction of sections from $\widetilde{X}$ to $X_1$ and $\gamma_{t}$ the trace map which restricts a section to the slice which is distant $t$ from the boundary. It is continuous and adjointable from the Sobolev modules $\mathcal{H}^s$ to $\mathcal{H}^{s-1/2}$ for $s>1/2$. 
In particular we have the trace to the boundary
$\gamma_0^-:\mathcal{H}^s(\widetilde{X};\widetilde{\mathcal{E}}^-\otimes \widetilde{\mathcal{A}})\longrightarrow \mathcal{H}^{s-1/2}(Y;\widetilde{\mathcal{E}}^-\otimes \widetilde{\mathcal{A}}_{|Y}).$ 
It is known that $1/2$ is the critical regularity for the trace to the boundary indeed the space of smooth sections supported in the interior is dense in $\mathcal{H}^{1/2}$. Instead solutions of the Dirac operators of any regularity  have traces.
\begin{teo}\label{traccia}For every positive real $s$ (actually for every) the trace map $\gamma$ is well defined as a map  $$\gamma:\operatorname{Ker}_{1,2}(\widetilde{D}^+,s)\longrightarrow \mathcal{H}^{s-1/2}(Y;\widetilde{\mathcal{E}}\otimes \widetilde{\mathcal{A}}_{|Y}),$$ in other words
$$\gamma(f)=\lim_{t \rightarrow 0}\gamma_{t}(f).$$
\end{teo}
\begin{proof}
Thanks to the complex interpolation procedure above defined the classical proof \cite{Boos} repeats words by words.
\end{proof}
Now define the \emph{Poisson operator}
$$K_1:=r_1 (\widetilde{D}^+)^{-1}(\gamma_0^-)^*G:\Gamma(Y;\widetilde{\mathcal{E}}^+\otimes \widetilde{\mathcal{A}}_{|Y})\longrightarrow \Gamma(Y;\widetilde{\mathcal{E}}^+\otimes \widetilde{\mathcal{A}}_{|X_1\setminus Y}).$$
\begin{teo}
$$ $$
The Poisson operator $K_1$ extends to a continuous surjective map from $\mathcal{H}^{s-1/2}(Y;\widetilde{\mathcal{E}}^+\otimes \widetilde{\mathcal{A}}_{|Y})$ to the $s$-- kernel $\operatorname{Ker}_{1}(\widetilde{D}^+,s)$ and induces by restriction a bijection from the Cauchy data space to the kernel:
$$K_1:H_{1}(\widetilde{D}^+,s)\longrightarrow \operatorname{Ker}_{1}(\widetilde{D}^+,s).$$
\end{teo}
\begin{proof}
First of all, as a consequence of \ref{traccia}, $K_1$ maps $\mathcal{H}^s$ continuously to $\mathcal{H}^{s+1/2}$ for $s\geq 0$. The Cauchy data space property is a straightforward computation as in \cite{Boos}.
\end{proof}
Now the Calderon projection is finally defined as
$$\mathcal{C}_{+}:=\gamma^+ K_1.$$ It maps $\mathcal{H}^{s}(Y;\widetilde{\mathcal{E}}^+\otimes \widetilde{\mathcal{A}}_{|Y})$ continuously to itself for every $s$, $s\geq 0$. The classical proof adapts to show that when restricted to $\mathcal{H}^0$ is a projection (idempotent) on the space of Cauchy datas $H_1$ along $H_2$. Similar statements hold reversing the role of the left/right side of the doubled manifold. 
\begin{teo}The Calderon projection is a pseudo differential operator of zero order in the Mischenko--Fomenko calculus on the boundary. 
\end{teo}
\begin{proof}The proof in \cite{Boos} works in the Mischenko--Fomenko framework. Due to its relevance we give a sketch. It is basically based on the main feature of the pseudodifferential calculus i.e the composition formula for symbols.
Localizing the problem the nontrivial step is the investigation of the limit
$$\lim_{t  \rightarrow 0}\gamma_{t}M_{\varphi_1}(\widetilde{D}^+)^{-1}M_{\varphi_2}\gamma_0^*,$$ where $M_{\varphi_1}$ are cutoff functions with non disjoint supports meeting the boundary and we have suppressed the suffix ${\pm}$ on the trace operators. The operator $M_{\varphi_1}(\widetilde{D}^+)^{-1}M_{\varphi_2}$ is pseudodifferential with total symbol admitting an asymptotic expansion 
$$\sum_{k=1}^{\infty}c_k,\quad c_{-1}=a_1^{-1}=\sigma_1(\widetilde{D}^+)^{-1}.$$ Then if the $\mathcal{C}_k=M_{\varphi_1}C_k(x,D)M_{\varphi_2}$ are the homogeneous pseudodifferential operators given by the symbols $c_k$ then we can write for every $k_0$
$$M_{\varphi_1}(\widetilde{D}^+)^{-1}M_{\varphi_2}=T+\underbrace{\sum_{k<k_0}\mathcal{C}_{-k}}_{\textrm{errors}}$$ with $T$ an operator of order $-K_0$ then only the investigations of the errors is needed. The point is a precise information contained in the symbol of $\mathcal{C}_{-k}\gamma_0^*$. Indeed for a test section $g$,
\begin{equation}\label{cald1}
\mathcal{C}_{-k}\gamma_0^*g(y,t)=\lim_{m\rightarrow \infty}(2\pi)^{-n}\int_{\RR^{n-1}}e^{iy\cdot \eta}\hat{g}(\eta)\bigg{[}\int_{-\infty}^{\infty}
e^{it\tau}\hat{\alpha}(\tau/m)c_{-k}(y,t;\eta,\tau)d\tau\bigg{]}d\eta
\end{equation} for a bump function $\alpha$ on $\RR$ supported in $(-1,-1/2)$ with $\int \alpha=1$. The decisive step is the replacement of the $\RR$ integral in \eqref{cald1} to an integral on a finite path $\Gamma(\eta)$ contained in $\Im(\tau)\geq 0$. This will provide uniform boundedness of the integrals permitting the passage of the limit inside the integral,
\begin{equation}\label{cald2}
\mathcal{C}_{-k}\gamma_0^*g(y,t)=(2\pi)^{-n}\int_{\RR^{n-1}}e^{iy\cdot \eta}\hat{g}(\eta)\bigg{[}\int_{\Gamma(\eta)}
e^{it\tau}c_{-k}(y,t;\eta,\tau)d\tau\bigg{]}d\eta.
\end{equation}
The choice of $\Gamma(\eta)$ 
depends on the invertible properties of the complex extension in $\tau$ of $a_1(y,t;\eta,\tau)$. In our case, the operator is a twisted Dirac operator and the principal symbol is the Clifford multiplication $a_1(\xi)= c(\xi)\otimes \operatorname{1_{\mathcal{A}}}$, so all the arguments in \cite{Boos} remain valid. More precisely for $\eta$ in the cosphere bundle there is a compact set $\mathcal{Z}\subset \mathbb{C}$ not intersecting the real axis such that $a_1$ is invertible for $\tau$ in the complement of $\mathcal{Z}$. Divide $\mathcal{Z}$ in the pieces $\mathcal{Z}^{\pm}$ with positive/negative imaginary part. 
The positive part is contained in a circle of radius $R$.
The contour is
$$\Gamma(\eta):=\partial\big{(}\{\tau:|\tau|\leq \max(1;R)\}\cap \{\Im \tau\geq 0\}\big{)}.$$

\begin{figure}[hpbt]
  \centering
  \includegraphics[width=4in]{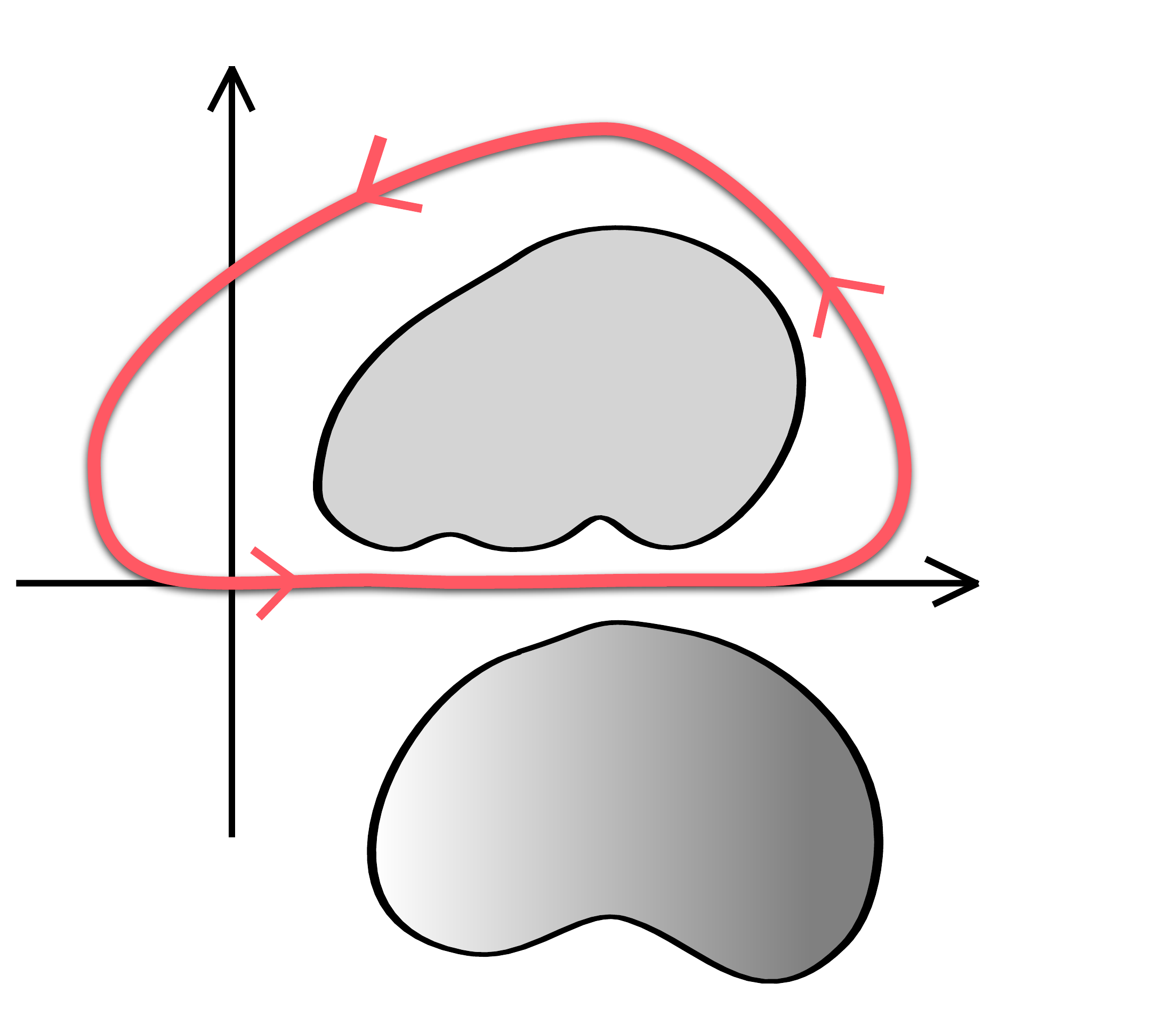}
\end{figure}

 Once the integral is changed into a contour integral one has dominated convergence with the integrands dominated by a common polynomial in $|\tau|$ then one can pass the limit inside the integral giving the uniform convergence of $\gamma_{t}\mathcal{C}_{-k}\gamma_0^*g$ in $L^2$. It also follows that the operator $\lim_{t \rightarrow 0}\gamma_{t}\mathcal{C}_{-k}\gamma_0^*$ is pseudodifferential of order $1-k$ with total symbol
$$p_{-k}(y,\eta):=1/2\pi \int_{\Gamma(\eta)}c_{-k}(y,0;\eta,\tau)d\tau$$ belonging to the homogenous standard symbol class.
\end{proof}
We can as well compute the principal symbol of the Calderon projection. Remember the product form near the boundary
$$\widetilde{D}^+=G(y)(\partial_u+B)$$ with selfadjoint boundary operator $B$  itself a twisted Dirac operator with twisted symbol $b(x,y)\otimes \mathbb{1}_{\mathcal{A}}:\widetilde{\mathcal{E}^+}_y\otimes \mathcal{A}_y\rightarrow \widetilde{\mathcal{E}^+}_y\otimes \mathcal{A}_y .$
Then
\begin{eqnarray}
\nonumber
\sigma_1(\mathcal{C}_+)(y,\eta)=\sigma_1(\lim \gamma_{t}\mathcal{C}_1\gamma_0^*G)(y,\eta)&=&1/2\pi\int_{\Gamma(\eta)}a_1^{-1}(0,y;\tau,\eta)d\tau \cdot G(y)\\ \nonumber&=&1/2\pi \int_{\Gamma(\eta)}\{G(y)(i\tau+b(x,y)\otimes \mathbb{1}_{\mathcal{A}})\}^{-1}\cdot G(y)d\tau\\ \nonumber
&=&-1/2\pi \int_{\Gamma(\eta)}\{
ib(x,y)\otimes \mathbb{1}_{\mathcal{A}}-\tau\}^{-1}dt\\
\nonumber
&=&q_+(x,y)\otimes \mathbb{1}_{\mathcal{A}}
,
\end{eqnarray}
where $q_+(x,y)$ is the spectral projection of $b(x,y)$ on the space associated to the eigenvalues with positive real part. 
Of course $\mathcal{C}_+$ is not self adjoint i.e. is not a orthogonal projection in the Hilbert module of $L^2$--sections on the boundary. Its range as a bounded operator in $L^2$ is closed and complementable being the range of a selfadjoint adjointable (since pseudofferential) idempotent. In the applications, especially dealing with spectral sections and global boundary value problems one can change the operator with its orthogonalized.
\begin{lem}The operator $F:=\mathcal{C}_+\mathcal{C}_+^*+(1-\mathcal{C}_+^*)(1-\mathcal{C}_+)$ acting on the Hilbert module of $L^2$--sections on the boundary is invertible. The operator $$\mathcal{C}_+^\bot:=\mathcal{C}_+\mathcal{C}_+^*F^{-1}$$ is the orthogonal projection on the image of $\mathcal{C}_+$. It is a pseudodifferential operator of order zero with the same principal symbol as $\mathcal{C}_+$.
\end{lem}
\begin{proof}
Since the range of the Calderon projection is complementable we have the splitting
$$\mathcal{H}=\operatorname{range}(\mathcal{C}_+)\oplus \operatorname{range}(\mathcal{C}_+)^{\bot}=\operatorname{range}(\mathcal{C}_+)\oplus \operatorname{Ker}(\mathcal{C}_+)^{*}$$
The operator $F$ can be written as
$$F=\underbrace{\mathcal{C}_+\mathcal{C}_+^*}_{G}+\underbrace{(1-\mathcal{C}_+^*)(1-\mathcal{C}_+)}_T$$ where $G$ is the partial isometry $\operatorname{range}(\mathcal{C}_+)\longrightarrow \operatorname{range}(\mathcal{C}_+^*)$ and $T$ is the partial isometry on the complements.
The rest of the proof is standard.
\end{proof}

\appendix

 \section{Hilbert Modules}
 We recall here some results on the theory of operators on Hilbert $C^*$--modules. For more details and proofs the reader is referred to \cite{lance, Skand, Wo}
\subsection{Basic Definitions} 
Let $A$ be a $C^*$ algebra. We denote by $P^+(A)$ the set of all positive elements in the $C^*$ algebra, i.e. those $a \in A$ satisfying one of the following equivalent conditions:
\begin{itemize}
\item $a$ has positive spectrum, i.e. $\sigma(a) \subset [0, +\infty)$;
\item $a=bb^*$ for some $b \in A$;
\item $a=h^2$ for some Hermitean $h\in A$.
\end{itemize}

We denote by $\mathcal{P}(A)$ the category of \emph{finitely generated projective modules} over $A$.
\begin{ddef}
A topological $A$-module $M$ is called a \emph{Hilbert $A$-modules} if it is equipped with a continuous map 
\[\begin{aligned}
M\times M \rightarrow A
(x,y) \mapsto \langle x , y \rangle_A
\end{aligned}\]
satisfying the conditions
\begin{enumerate}
\item $\langle x, x \rangle_A \geq 0$ for any $x \in M$, i.e. $\langle x, x \rangle_A \in P^+(A)$;
\item $\langle x, x \rangle_A =0$ if and only if $x=0$;
\item $\langle x, y \rangle_A = \langle y, x \rangle_A^*$ for all $x, y \in M$;
\item $\langle x, ya \rangle_A =\langle x, y \rangle_A$ for all $x, y \in M$ and $a \in A$.
\end{enumerate}
and the module $M$ is a Banach space with respect to the norm induced by the inner product $\| x \|^2_A :=\| \langle x, x \rangle_A \|_A$, where $\| \cdot \|_A$ denotes the $C^*$-norm on $A$. 
\end{ddef} 
The map $\langle \cdot, \cdot \rangle_A$ is defined to be a Hermitean $A$-inner product. When no confusion arises we will omit the subscript. 

We will say that two elements $x$ and $y$ in $M$ are orthogonal if $\langle x, y \rangle =0 $, and we will write $x \perp y$. Orthogonality for arbitrary sets may be defined similarly.

If $N, L$ are two closed submodules in $M$, and $N\oplus L = M$, $N \perp M$, then $N$ is called the $A$-orthogonal complement of $L$ in $M$ (and viceversa).

Any free $A$-module $A^n$ and any projective module $P \in \mathcal{P}(A)$ can be equipped with a structure of Hilbert $A$-module. 

Let $P \in \mathcal{P}(A)$. We consider the set $\ell_2(P)$ of infinite sequences
\[(x_1, \dots, x_i, \dots) = x \quad x_i \in P, \quad i=1, \dots\]
such that the series $\sum_{i} \langle x_i , x_i \rangle_A$ converges in the algebra $A$. For any two elements $x, y \in P$ we define 
\[\langle x, y \rangle = \sum_{i=0}^{\infty} \langle x_i, y_i \rangle_A .\]
The space $\ell^2(P)$ is a Hilbert $A$-module with respect to this inner product.
\begin{lem}
Any free $A$-module admits a unique (up to isomorphism) $A$-inner product.
\end{lem}
\begin{teo}
Suppose that $\ell_2(A)= M \oplus N$, where $M, N$ are closed submodules of $\ell_2(A)$ and $N$ has a finite number of generators. Then $N$ is a projective module.
\end{teo}

\subsection{Operators in Hilbert Modules}
Let $M, N$ be Hilbert modules. 
\begin{ddef}
A mapping $T: M \rightarrow N$ such that for some mapping $T^*: N \rightarrow M$ the relation
\begin{equation}
\langle u, Tv \rangle = \langle T^* u, v \rangle \qquad \mbox{holds for all} \ u \in N, v \in M,
\end{equation}
is called an operator from $M$to $N$. We denote the space of such operatos by $\mathcal{L}(M,N)$.
\end{ddef} Note that for mappings between $C^*$-modules, the conditions of linerarity and boundedness do not always imply the existence of an adjoint.

The subspace $\mathcal{K}(M,N) \subset \mathcal{L}(M,N)$ of compact operators is determined as the norm closure of the space generated by operators of rank 1, i.e. operators of the form \[\theta_{x,y} z:= x \langle y, z \rangle \qquad x \in M_2, \quad y, z \in M_1.\]
Obviously, \[(\theta_{x,y})^* = \theta_{y,x} \qquad \theta_{x,y}\theta_{u,v} = \theta_{x \langle y, u \rangle, v}=\theta_{x, v \langle u, y \rangle}
\]  
For $M=N$, teh space $\mathcal{K}(M) := \mathcal{K}(M,M)$ is a $C^*$-ideal in $\mathcal{L}(M)$.
\begin{lem}
 Let $T\in \mathcal{L}(M,N)$ 
 \begin{enumerate}
 \item if $T$ is surjective $TT^*$ is invertible in $\mathcal{L}(M)$ and $M=\operatorname{Ker}T\oplus \operatorname{Ran}T^*.$
 \item If $T$ is bijective then so it is $T^*$, $T^{-1}\in \mathcal{L}(N,M)$ and $(T^{-1})^*=(T^*)^{-1}.$
 \end{enumerate}
 \end{lem}
 \begin{lem}Let $T\in \mathcal{L}(M,N).$ The following conditions are equivalent:
 \begin{itemize}
 \item $\operatorname{Ran}T$ is closed in $N$,
 \item $\operatorname{Ran}T^*$ is closed in $M$,
 \item $0$ is isolated in the spectrum of $T^*T$,
 \item $0$ is isolated in the spectrum of $TT^*$.
 \end{itemize}
  \end{lem}
 \begin{lem}[Mishchenko]
 Let $M,N$ Hilbert $A$--modules and $T\in \mathcal{L}(M,N)$ an operator with closed range. Then $\operatorname{Ker}T$ is complemented in $M$, $\operatorname{Ran}T$ is complemented in $N$ and
 $$N=\operatorname{Ker}T^*\oplus \operatorname{Ran}T.$$
 \end{lem}
 \begin{proof}Let $N_0:=\operatorname{Ran}T$ and $T_0:M\longrightarrow N_0$ an operator such its action coincides with the action of $T$. By the open mapping theorem $T_0(B_1(M))$ (the image of the unit ball) contains some ball of radius $\delta>0$ in $N_0$. Therefore for every $y\in N_0$ there is an $x\in M$ such that $T_0x=y$ and $\|x\|\leq \delta^{-1}\|y\|.$
 Then $\|T^*_0y\|^2=\|\langle y, T_0T^*_0y\rangle \|\leq \|y\|\cdot \|T_0T_0^* y\|$ hence 
 $$\|y\|^2=\|\langle T_0x,y\rangle \|=\|\langle x,T_0^*y\|\leq \|x\|\cdot \|T^*_0y\|\leq \delta^{-1}\|y\|^{3/2}\|T_0T_0^*y\|^{1/2},$$
 i.e. $$\|Y\|\leq \delta^{-2}\|T_0T_0^*y\|,\quad y\in N_0.$$ Now since $0$ is not in the spectrum of $T_0T_0^*$ we have that $T_0T_0^*$ is invertible and for every $z\in M$ there exists $w\in N_0$ such that $T_0z=T_0T_0^*w.$ Then $z-T_0^*w \in \operatorname{Ker}T$ and
 $$z=(z-T_0^*w)+T_0^*w\in \operatorname{Ker}T+\operatorname{Ran}T_0^*.$$ 
 Since $\operatorname{Ran}T_0^*$ is obviously orthogonal to 
 $\operatorname{Ker}T$, it is a complement for 
 $\operatorname{Ker}T$. 
 This completes the proof of orthogonal complementability for $\operatorname{Ker}T$. Now pass to 
 $\operatorname{Ran}T$. Since 
 $M=\operatorname{Ker}T\oplus \operatorname{Ran}T_0^*$, 
 the submodule 
 $\operatorname{Ran}T_0^*$ is closed. Note that 
 $\operatorname{Ran}T_0^*=\operatorname{Ran}T^*$ so one can apply the previous argument to $T^*$ instead of $T$ which gives the orthogonal decomposition
 $$N=\operatorname{Ker}T^*\oplus \operatorname{Ran}T.$$
 \end{proof}

\bibliographystyle{abbrv}
\bibliography{bibCalderon}
\end{document}